\documentclass{lmcs}
\pdfoutput=1

\usepackage[utf8]{inputenc}

\usepackage{lastpage}
\lmcsdoi{17}{1}{3}
\lmcsheading{}{\pageref{LastPage}}{}{}%
{Sep.~14,~2020}{Jan.~22,~2021}{}


\usepackage{amsmath, amssymb, amsthm, stmaryrd, mathabx}
\usepackage{thmtools}

\usepackage{tikz-cd}

\begin{document}

\title{2-adjoint equivalences in homotopy type theory}

\author[D.~Carranza]{Daniel Carranza}
\address{Department of Mathematics, University of Western Ontario}
\urladdr{http://daniel-carranza.github.io/}
\email{dcarran@uwo.ca}

\author[J.~Chang]{Jonathan Chang}
\address{Department of Combinatorics and Optimization, University of Waterloo}
\email{jchan767@uwo.ca}

\author[K.~Kapulkin]{Krzysztof Kapulkin}
\address{Department of Mathematics, University of Western Ontario}
\email{kkapulki@uwo.ca}

\author[R.~Sandford]{Ryan Sandford}
\address{Department of Computer Science, University of Western Ontario}
\email{rsandfo@uwo.ca}

\begin{abstract}
	\noindent We introduce the notion of (half) 2-adjoint equivalences in Homotopy Type Theory and prove their expected properties.
	We formalized these results in the Lean Theorem Prover.
\end{abstract}

\maketitle
	
	\section*{Introduction}

There are numerous notions of equivalence in homotopy type theory: bi-invertible maps, contractible maps, and half adjoint equivalences.
Other natural choices, such as quasi-invertible maps and adjoint equivalences, while \emph{logically} equivalent to the above, are not propositions, making them unsuitable to serve as \emph{the} definition of an equivalence.
One can use a simple semantical argument, which in essence comes down to analyzing subcomplexes of the nerve of the groupoid $(0 \cong 1)$, to see why some definitions work and others do not.
The conclusion here is that while the definition as a ``half $n$-adjoint equivalence'' gives us a proposition, the definition as a ``(full) $n$-adjoint equivalence'' does not.

In this paper, we take the first step towards expressing these results internally in type theory, putting special emphasis on their formalization.
In particular, we revisit the notions of a quasi-invertible map, a half adjoint equivalence, and an adjoint equivalence, giving the formal proofs of their expected properties.
Our proofs are more modular than those given in \cite{HoTT:book}, and help improve efficiency.
We then turn our attention to corresponding notions arising from $2$-adjunctions, namely half $2$-adjoint equivalences and (full) $2$-adjoint equivalences, and show that while the former is always a proposition, the latter fails to be one in general.
	 
As indicated above, the results proven here certainly will not come as
a surprise to experts and they constitute merely the first step
towards understanding general $n$-adjoint equivalences.
One can therefore envision future work in which the notions of $3$-,
$4$-, \ldots and, more generally, $n$-adjoint equivalence are studied.
We have chosen not to pursue this direction, simply because the
corresponding notions of $3$-, $4$-, and $n$-adjunction have not ---
to our knowledge --- received rigorous treatment in literature on
category theory.
In particular, it is not immediately clear what the higher-dimensional
analog of the coherences appearing in the swallow-tail identities
ought to be (cf.~$\mathsf{Coh} \: \eta$ in Definition \ref{D:two_hae}).
Having said that, we believe that the approach developed here can
serve as a blueprint for proving analogous properties of $n$-adjoint
equivalences when these notions are introduced.

	These results have been formalized using the Lean Theorem Prover, version 3.4.2
		(\url{https://github.com/leanprover/lean}) as part of the HoTT in Lean 3
		library (\url{https://github.com/gebner/hott3}); the formalization consists of 528 lines of code across 3 files
		and may be found in the directory
			\href{https://github.com/gebner/hott3/tree/master/src/hott/types/2_adj}
			{\texttt{hott3/src/hott/types/2\_adj}}.
		We write \texttt{file/name} for a newly-formalized result, where \texttt{file}
			denotes the file it is found in and \texttt{name} denotes the name of the formal proof in the code.
		
	\section*{Organization.} Section 1 recalls the necessary background on equivalences which
		will be used throughout. 
		Section 2 introduces new formal proofs that the types of quasi-inverses and  adjoint equivalences
		are not propositions. Note that specific examples where this fails are presented, but not formally
		proven since the current version of the HoTT in Lean 3 library does not contain
		induction principles for the higher inductive types $S^1$ and $S^2$.
		Section 3 introduces half 2-adjoint equivalences, which are propositions
		containing the data of adjoint equivalences, as well as  2-adjoint equivalences, which
		are non-propositions related to both quasi-inverses and adjoint equivalences.
		
	\section{Preliminaries}
	We largely adopt the notation of \cite{HoTT:book}, with additional and differing notation stated here.
	We notate the $\mathsf{ap}$ function for $f : A \to B$ by
	\[
		f[ - ] : (x = y) \to (fx = fy).
	\]
	For $\mathsf{ap}_2$, the action of $f$ on 2-dimensional paths, we write
	\[
		f \llbracket - \rrbracket : (p = q) \to (f[p] = f[q]).
	\]
	We write $\mathsf{refl}$ for the homotopy $\lambda x. \mathsf{refl}_x \colon \mathsf{id}_A \sim \mathsf{id}_A$.
	For a homotopy $H : f \sim g$ between dependent functions $f, g : \prod\limits_{x : A} B x$
	 	and a non-dependent function $h : C \to A$, we write
	\[
		H_h : fh \sim gh
	\]
	for the composition of $H$ and $h$. If $f, g : A \to B$ are non-dependent and we instead have $h : B \to C$,
		we write
	\[
		h[H] : hf \sim hg
	\]
	for the composition of $h[-]$ and $H$. Given an additional homotopy $H' : f \sim g$ and $\alpha : H \sim H'$,
		we similarly write
	\[
		h \llbracket \alpha \rrbracket : h[H] \sim h[H']
	\]
	for the composition of $h \llbracket - \rrbracket$ and $\alpha$. Lastly,
	for $H : f \sim g$ and $H' : g \sim h$, we write transitivity of homotopies as
	\[
		H \cdot H' : f \sim h
	\] 
	in path-concatenation order.
		
	\begin{defi}[\texttt{adj/qinv}, \texttt{adj/is\_hadj\_l}]
		A function $f : A \to B$ \begin{enumerate}
			\item has a \emph{quasi-inverse} if the following type is inhabited:
		\[
			\mathsf{qinv}\: f :\equiv \sum_{g : B \to A} gf \sim \mathsf{id}_A
				\times fg \sim \mathsf{id}_B.
		\]
			\item is a \emph{half-adjoint equivalence} if the following type is inhabited:
		\[
			\mathsf{ishadj} \: f :\equiv \sum_{g : B \to A} \  \sum_{\eta : gf \sim \mathsf{id}_A} \ 
				\sum_{\varepsilon : fg \sim \mathsf{id}_B} f[\eta] \sim \varepsilon_f.
		\]
			\item is a \emph{left half-adjoint equivalence} if the following type is inhabited:
		\[
			\mathsf{ishadjl}\: f :\equiv \sum_{g : B \to A} \  \sum_{\eta : gf \sim \mathsf{id}_A} \ 
				\sum_{\varepsilon : fg \sim \mathsf{id}_B} \eta_{g} \sim g[\varepsilon].
		\]
		For types $A, B : \mathcal{U}$, the type of equivalences between $A$ and $B$ is:
		\[
			A \simeq B :\equiv \sum_{f : A \to B} \mathsf{ishadj} \: f.
		\]
		\end{enumerate}
	\end{defi}
	\begin{thmC}[{\cite[Lem.~4.2.2, Thms.~4.2.3,~4.2.13]{HoTT:book}}]
		For $f : A \to B$, there are maps
		\[
			\begin{tikzcd}[column sep=small]
				\mathsf{ishadj} \: f \ar[rr, leftrightarrow, "\simeq"] \ar[rd, leftrightarrow] 
					&& \mathsf{ishadjl} \: f \ar[ld, leftrightarrow] \\
				& \mathsf{qinv} \: f
			\end{tikzcd}
		\]
		where the top two types are propositions. \qed
	\end{thmC}

	The perhaps most intuitive definition of an equivalence between types $A, B : \mathcal{U}$ is that of a quasi-inverse. 
		However, as this type is not a proposition, we 
		define equivalences to be half adjoint equivalences. Since both half and left half adjoint equivalences 
		are propositional types, one could also define
		the type of equivalences to be left half adjoint equivalences.
	
	With a well-behaved notion of equivalence, we present the remaining lemmas to be used throughout.
	\begin{lem}[Equivalence Induction/Univalence, {\cite[Cor.~5.8.5]{HoTT:book}}] \label{T:equivInduction}
		Given $D : \prod_{A, B : \mathcal{U}} (A \simeq B) \to \mathcal{U}$ and
			$d : \prod_{A : \mathcal{U}} D(A, A, \mathsf{id}_A)$, there exists 
		\[
			f : \prod_{A, B : \mathcal{U}} \  \prod_{e : A \simeq B} D(A, B, e)
		\]
		such that $f(A, A, \mathsf{id}_A) = d(A)$ for all $A : \mathcal{U}$. \qed
	\end{lem}
	\begin{lem}[\texttt{prelim/sigma\_hty\_is\_contr}, {\cite[Cor.~5.8.6, Thm.~5.8.4]{HoTT:book}}]
		\label{T:isContrSigma}
		Given $f : A \to B$, the types 
		\[
			\sum_{g : B \to A} f \sim g \text{ and } \sum_{g : B \to A} g \sim f
		\] 
		are both contractible with center $(f, \mathsf{refl}_{f})$. \qed
	\end{lem}
	
	\begin{lemC}[{\cite[Lem.~4.2.5]{HoTT:book}}] \label{T:fibeqchar}
		For any $f : A \to B$, $y : B$ and $(x, p), (x', p') : \mathsf{fib}_f \: y$, we have
		\[
			(x, p) = (x', p') \simeq \sum_{\gamma : x = x'} p = f[\gamma] \cdot p'.
      \tag*{\qed}
		\]
	\end{lemC}
	\begin{lemC}[{\cite[Thm.~4.2.6]{HoTT:book}}] \label{T:fibcontr}
		If $f : A \to B$ is a half-adjoint equivalence, then for any $y : B$ the fiber $\mathsf{fib}_f \: y$
			is contractible. \qed
	\end{lemC}
	
	\section{Quasi-inverses and adjoint equivalences, revisited}
	
	We present a proof that the type of quasi-inverses is not a proposition, using
		Lemma \ref{T:isContrSigma} for increased modularity over the proof presented in \cite[Lem.~4.1.1]{HoTT:book}.
	\begin{thm}[\texttt{adj/qinv\_equiv\_pi\_eq}] \label{T:qinvpi}
		Given $f : A \to B$ such that $\mathsf{ishadj}\: f$ is inhabited, we have
		\[
			\mathsf{qinv}\: f \simeq \prod_{x : A} x = x.
		\]
	\end{thm}
	\begin{proof}
		By Equivalence Induction \ref{T:equivInduction}, 
			it suffices to show $\mathsf{qinv} \: \mathsf{id}_A \simeq 
			\prod\limits_{x : A} x = x$. Observe that
		\begin{align}
			\mathsf{qinv} \: \mathsf{id}_A &\equiv \sum_{g : A \to A} 
				g  \sim \mathsf{id}_A \times
				g \sim \mathsf{id}_A \nonumber \\
			&\simeq \sum_{g : A \to A} \  \sum_{\eta : g \sim \mathsf{id}_A}
				g \sim \mathsf{id}_A \nonumber \\
			&\simeq \sum_{u : \sum\limits_{g : A \to A} 
				g \sim \mathsf{id}_A} 
				\mathsf{pr}_1 \: u \sim \mathsf{id}_A \nonumber \\
			\label{eq:qinvline} &\simeq \mathsf{id}_A \sim \mathsf{id}_A  \\
			&\equiv \prod_{x : A} x = x, \nonumber
		\end{align}
		where (\ref*{eq:qinvline}) follows from Lemma \ref{T:isContrSigma} (the type $\sum\limits_{g : A \to A} 
			g \sim \mathsf{id}_A$ is contractible with center $(\mathsf{id}_A, \mathsf{refl})$).
	\end{proof}
	This result implies that any type with non-trivial $\pi_1$ 
			may be used to construct non-trivial inhabitants of this type. 
			For instance, since $\pi_1(S^1) = \mathbb{Z}$, we have:
	\begin{cor}
		The type $\mathsf{qinv} \: \mathsf{id}_{S^1}$ is not a proposition. \qed
	\end{cor}
	Conceptually, this proof takes the pair $(g, \eta)$ and uses Lemma \ref{T:isContrSigma} to contract it
		so that only one homotopy remains.  
		This differs from the proof in \cite{HoTT:book}, which uses function extensionality
		to write the homotopies as paths and contracts using based path induction. 
		This proof modularizes the proof in \cite{HoTT:book} by packaging function
		extensionality and rewriting of contractible types into one result, simplifying both the proof and the
		formalization.
		
	Thus, the type of half and left half adjoint equivalences each append an additional coherence 
		to contract with the remaining homotopy. However, appending both coherences
		gives us a non-proposition.
	\begin{defi}[\texttt{adj/adj}]
		Given $f : A \to B$, the structure of an \emph{adjoint equivalence} on $f$ is the type:
		\[
			\mathsf{adj}\: f :\equiv \sum_{g : B \to A} \  \sum_{\eta : gf \sim \mathsf{id}_A} \ 
				\sum_{\varepsilon : fg \sim \mathsf{id}_B} \ 
				f[\eta] \sim \varepsilon_f \times \eta_g \sim g[\varepsilon].
		\]
	\end{defi}
	\begin{thm}[\texttt{adj/adj\_equiv\_pi\_refl\_eq}] \label{T:adjequiv}
		Given $f : A \to B$ such that $\mathsf{ishadj}\: f$ is inhabited, we have
		\[
			\mathsf{adj}\: f \simeq \prod_{x : A} (\mathsf{refl}_x = \mathsf{refl}_x).
		\]
	\end{thm}
	\begin{proof}
		By Equivalence Induction \ref{T:equivInduction}, it suffices to show $\mathsf{adj} \: \mathsf{id}_A 
			\simeq \prod\limits_{x : A} \mathsf{refl}_x = \mathsf{refl}_x$. Observe that
		\begin{align}
			\mathsf{adj} \: \mathsf{id}_A &\equiv \sum_{g : A \to A} \ 
				\sum_{\eta : g \sim \mathsf{id}_A} \ 
				\sum_{\varepsilon : g \sim \mathsf{id}_A}
				\mathsf{id}_A[\eta] \sim \varepsilon \times \eta_g \sim g[\varepsilon] \nonumber \\
			\label{eq:adjprev} &\simeq \sum_{\varepsilon : \mathsf{id}_A \sim \mathsf{id}_A}
				\mathsf{refl} \sim \varepsilon \times \mathsf{refl} \sim \mathsf{id}_A[\varepsilon] \\
			&\simeq \sum_{\varepsilon : \mathsf{id}_A \sim \mathsf{id}_A} \
				\sum_{\tau : \mathsf{refl} \sim \varepsilon} 
				\mathsf{refl} \sim \mathsf{id}_A[\varepsilon] \nonumber \\
			&\simeq \sum_{u : \sum\limits_{\varepsilon : \mathsf{id}_A \sim \mathsf{id}_A} 
					\mathsf{refl} \sim \varepsilon}
				\mathsf{refl} \sim \mathsf{id}_A[\mathsf{pr}_1 \: u] \nonumber \\
			\label{eq:adjline} &\simeq \mathsf{refl} \sim \mathsf{id}_A [\mathsf{refl}] \\
			&\equiv \prod_{x : A} (\mathsf{refl}_x = \mathsf{refl}_x). \nonumber
		\end{align}
		The equivalence (\ref*{eq:adjprev}) comes from the equivalence in Theorem \ref{T:qinvpi}, where the pair
			$(g, \eta)$ contracts to $(\mathsf{id}_A, \mathsf{refl})$. The equivalence (\ref*{eq:adjline})
			follows from Lemma \ref{T:isContrSigma}.
	\end{proof}
	This result implies that any type with non-trivial $\pi_2$ may be used to construct
		non-trivial inhabitants of this type. In particular, $\pi_2(S^2) = \mathbb{Z}$ proves the following:
	\begin{cor}
		The type $\mathsf{adj} \: \mathsf{id}_{S^2}$ is not a proposition. \qed
	\end{cor}
	This is a solution to Exercise 4.1 in \cite{HoTT:book}.  
		As before, this proof uses Lemma \ref{T:isContrSigma} to contract the pairs 
		$(g, \eta)$ and $(\varepsilon, \tau)$ so that a single homotopy remains. 
		Trying to apply path induction directly requires an equivalence which writes each homotopy as an equality; 
		a formal proof using function extensionality for such an equivalence 
		along with path induction reaches 60 lines of code (varying by format, syntax choice, etc.). 
		By modularizing the case of $\mathsf{qinv}$, this proof is reduced to manipulating $\Sigma$-types and
		applying Lemma \ref{T:isContrSigma} twice, with the formal proof in the library being 23 lines of code.

	\section{2-adjoint equivalences}
	As in the case of $\mathsf{qinv}$, we expect there is an additional coherence that may be
		appended to the type $\mathsf{adj}\: f$ to create a proposition. 
		To define this coherence, we use the following homotopy:
	\begin{lem}[\texttt{two\_adj/nat\_coh}] \label{T:nat_coh}
		Given $f : A \to B$ and $g : B \to A$ with a homotopy $H : gf \sim \mathsf{id}_A$,
			we have a homotopy
		\[
			\mathsf{Coh} \: H : H_{gf} \sim g[f[H]]
		\]
		such that
		\[
			\mathsf{Coh} \: \mathsf{refl} \equiv \mathsf{refl}_{\mathsf{refl}} : \mathsf{refl} \sim \mathsf{refl}.
		\]
	\end{lem}
	\begin{proof}
		Fix $x : A$. We have
		\[ 
			\begin{split}
				H_{g(fx)} & = (gf)[H_x] \\
				& = g[f[H_x]],
			\end{split}
		\]
		where the first equality holds by naturality and the second holds by functoriality of $g[-]$.
	\end{proof}
	With this, we define the type of half 2-adjoint equivalences.
	\begin{defi}[\texttt{two\_adj/is\_two\_hae}] \label{D:two_hae}
		A function $f : A \to B$ is a \emph{half 2-adjoint equivalence} 
			if the following type is inhabited:
		\[
			\mathsf{ish2adj}\: f :\equiv \sum_{g : B \to A} \  \sum_{\eta : gf \sim \mathsf{id}_A} \ 
				\sum_{\varepsilon : fg \sim \mathsf{id}_B} \ 
				\sum_{\tau :f[\eta] \sim \varepsilon_{f}} \ 
				\sum_{\theta : \eta_{g} \sim g[\varepsilon]} \ 
				\mathsf{Coh} \: \eta \cdot g \llbracket \tau \rrbracket \sim \theta_f.
		\]
	\end{defi}
	In parallel with adjoint equivalences, we give a definition which uses an alternate coherence.
	\begin{defi}[\texttt{two\_adj/is\_two\_hae\_l}]
		A function $f : A \to B$ is a \emph{left half 2-adjoint equivalence} 
			if the following type is inhabited:
		\[
		\mathsf{ish2adjl}\: f :\equiv \sum_{g : B \to A} \  \sum_{\eta : gf \sim \mathsf{id}_A} \ 
				\sum_{\varepsilon : fg \sim \mathsf{id}_B}  \ 
				\sum_{\tau : f[\eta] \sim \varepsilon_{f}} \ 
				\sum_{\theta : \eta_{g} \sim g[\varepsilon]} \ 
				\tau_{g} \cdot \mathsf{Coh} \: \varepsilon \sim f \llbracket \theta \rrbracket.
		\]
	\end{defi}
	To show the type of half 2-adjoint equivalences is a proposition, we prove the following lemma:
	\begin{lem}[\texttt{two\_adj/r2coh\_equiv\_fib\_eq}] \label{T:tofiber}
		Given $f : A \to B$ with $(g, \eta, \varepsilon, \theta) : \mathsf{ishadjl} \: f$, we have
		\[
			\sum_{\tau : f[\eta] \sim \varepsilon_{f}} 
				\mathsf{Coh} \: \eta \cdot g\llbracket \tau \rrbracket \sim \theta_f \\
				\simeq \prod_{x : A} \left( f[\eta_x], \mathsf{Coh} \: \eta_x \cdot \theta_{fx} \right)
				= \left( \varepsilon_{fx}, \mathsf{refl}_{g[\varepsilon_{fx}]} \right),
		\]
		where $\left( f[\eta_x], \mathsf{Coh} \: \eta_x \cdot \theta_{fx} \right), 
			\left( \varepsilon_{fx}, \mathsf{refl}_{g[\varepsilon_{fx}]} \right) : \mathsf{fib}_{g[-]}
			\: g[\varepsilon_{fx}]$.
	\end{lem}
	\begin{proof}
		We have
		\begin{align}
			\sum_{\tau : f[\eta] \sim \varepsilon_f} \mathsf{Coh} \: \eta \cdot g\llbracket \tau \rrbracket
				\sim \theta_f &\equiv 
			\sum_{\tau : \prod_{x : A} f[\eta_x] = \varepsilon_{fx}}  \ 
				\prod_{x : A} \mathsf{Coh} \: \eta_x \cdot g \llbracket \tau_x \rrbracket = \theta_{fx}, 
				\nonumber \\
			\label{eq:fibone} &\simeq \prod_{x : A} \ \sum_{\tau' : f[\eta_x] = \varepsilon_{fx}} 
				\mathsf{Coh} \: \eta_x \cdot g \llbracket \tau' \rrbracket = \theta_{fx} \\
			\label{eq:fibtwo} &\simeq \prod_{x : A} \  \sum_{\tau' : f[\eta_x] = \varepsilon_{f(x)}}
				\mathsf{Coh} \: \eta_x^{-1} \cdot \theta_{fx} = g \llbracket \tau' \rrbracket \\
			\label{eq:fibthree} &\simeq \prod_{x : A} \left( f[\eta_x], (N_\eta)^{-1} \cdot \theta_{f(x)} \right)
				= \left( \varepsilon_{f(x)}, \mathsf{refl}_{g[\varepsilon_{f(x)}]} \right).
		\end{align}
		The equivalence (\ref*{eq:fibone}) holds by the Type-Theoretic Axiom of Choice, 
			(\ref*{eq:fibtwo}) is a rearrangment of equality, and 
			(\ref*{eq:fibthree}) holds by Lemma \ref{T:fibeqchar}.
	\end{proof}
	\begin{lem}[\texttt{two\_adj/is\_contr\_r2coh}] \label{T:iscontrr2coh}
		Given $f : A \to B$ with $(g, \eta, \varepsilon, \theta) : \mathsf{ishadj} \: f$, the type
		\[
			\sum_{\tau : f[\eta] \sim \varepsilon_{f}}
				\mathsf{Coh} \: \eta \cdot g \llbracket \tau \rrbracket \sim \theta_f
		\]
		is contractible.
	\end{lem}
	\begin{proof}
		By Lemma \ref{T:tofiber} and contractibility of $\Pi$-types, it suffices to fix $x : A$ and show the type
		\[
			\left( f[\eta_x], \mathsf{Coh} \: \eta_x^{-1} \cdot \theta_{fx} \right) 
				= \left( \varepsilon_{fx}, \mathsf{refl}_{g[\varepsilon_{fx}]} \right)
		\]
		is contractible. Since $g$ is an equivalence, $g[-]$ is also an equivalence.
		 	By Lemma \ref{T:fibcontr}, the type $\mathsf{fib}_{g[-]}(g[\varepsilon_{fx}])$ 
		 	is contractible, so its equality type is also contractible.
	\end{proof}
	\begin{thm}[\texttt{two\_adj/is\_prop\_is\_two\_hae}] \label{T:isprop2adj}
		For any $f : A \to B$, the type $\mathsf{ish2adj}\: f$ is a proposition.
	\end{thm}
	\begin{proof}
		It suffices to assume $e : \mathsf{ish2adj}\: f$ and show this type is contractible.
			Observe that
		\[ \begin{split}
			\mathsf{ish2adj}\: f &\equiv \sum_{g : B \to A} \  \sum_{\eta : gf \sim \mathsf{id}_A} \ 
				\sum_{\varepsilon : fg \sim \mathsf{id}_B} \ 
				\sum_{\tau : f[\eta] \sim \varepsilon_{f}} \ 
				\sum_{\theta : \eta_{g} \sim g[\varepsilon]} 
				\mathsf{Coh} \: \eta \cdot g \llbracket \tau \rrbracket \sim \theta_{f} \\
			&\simeq \sum_{g : B \to A} \  \sum_{\eta : gf \sim \mathsf{id}_A} \ 
				\sum_{\varepsilon : fg \sim \mathsf{id}_B} \ 
				\sum_{\theta : \eta_{g} \sim g[\varepsilon]} \ 
				\sum_{\tau : f[\eta] \sim \varepsilon_{f}} 
				\mathsf{Coh} \: \eta \cdot g \llbracket \tau \rrbracket \sim \theta_{f} \\
			&\simeq
				\sum_{(g, \eta, \varepsilon, \theta) : \mathsf{ishadjl}\: f} \ 
				\sum_{\tau : f[\eta] \sim \varepsilon_{f}}
				\mathsf{Coh} \: \eta \cdot g \llbracket \tau \rrbracket \sim \theta_{f} \\
			&\simeq 
				\sum_{\tau : f[\eta_0] \sim (\varepsilon_0)_f}
				\mathsf{Coh} \: \eta_0 \cdot g_0 \llbracket \tau \rrbracket \sim (\theta_0)_f.
		\end{split} \]
		The last equivalence holds since $\mathsf{ishadjl} f$ is contractible (it is a proposition and inhabited by
			$e$ after discarding coherences); we write $(g_0, \eta_0, \varepsilon_0, \theta_0) : \mathsf{ishadjl} f$
			for its center of contraction. This final type is contractible by Lemma \ref{T:iscontrr2coh},
			therefore $\mathsf{ish2adj} f$ is contractible.
	\end{proof}
	Parallels of these proofs are used to obtain similar results about left half two-adjoint equivalences as well.
	\begin{lem}[\texttt{two\_adj/is\_contr\_l2coh}] \label{T:iscontrl2coh}
		Given $f : A \to B$ with $(g, \eta, \varepsilon, \tau) : \mathsf{ishadj} f$, the type
		\[
			\sum_{\theta : \eta_{g} \sim g[\varepsilon]}
				\tau_{g} \cdot \mathsf{Coh} \: \varepsilon \sim f \llbracket \theta \rrbracket.
		\]
		is contractible.
	\end{lem}
	\begin{proof}
		Analogous to Lemma \ref{T:iscontrr2coh}.
	\end{proof}
	\begin{thm}[\texttt{two\_adj/is\_prop\_is\_two\_hae\_l}]
		For $f : A \to B$, the type $\mathsf{ish2adjl}\: f$ is a proposition.
	\end{thm}
	\begin{proof}
		Analogous to Theorem \ref{T:isprop2adj}.
	\end{proof}
	As well, either half adjoint equivalence may be promoted to the alternate half 2-adjoint equivalence.
	\begin{thm}[\texttt{two\_adj/two\_adjointify}] \label{T:2promote}
		For $f : A \to B$, we have maps
		\begin{enumerate}
			\item $\mathsf{ishadjl} \: f \to \mathsf{ish2adj} \: f$
			\item $\mathsf{ishadj} \: f \to \mathsf{ish2adjl} \: f$
		\end{enumerate}
	\end{thm}
	\begin{proof}
		Take the missing coherences to be the centers of contraction 
			from Lemmas \ref{T:iscontrr2coh} and \ref{T:iscontrl2coh}.
	\end{proof} 
	This implies that an adjoint equivalence may be promoted to either half 2-adjoint equivalence.
	\begin{cor}
		For $f : A \to B$, we have maps
		\begin{enumerate}
			\item $\mathsf{adj} \: f \to \mathsf{ish2adj} \: f$
			\item $\mathsf{adj} \: f \to \mathsf{ish2adjl} \: f$
		\end{enumerate}
	\end{cor}
	\begin{proof}
		Discard either coherence and apply Theorem \ref{T:2promote}.
	\end{proof}
	Finally, we have that the half 2-adjoint and left half 2-adjoint equivalences are logically equivalent.
	\begin{thm}[\texttt{two\_adj/two\_hae\_equiv\_two\_hae\_l}]
		For $f : A \to B$, we have maps
		\[
			\mathsf{ish2adj}\: f \leftrightarrow \mathsf{ish2adjl}\: f.
		\]
	\end{thm}
	\begin{proof}
		In either direction, discard coherences and apply Theorem \ref{T:2promote}.
	\end{proof}
	We summarize the properties of these 2-adjoint equivalances with the following
		diagram of maps:
	\[
		\begin{tikzcd}[column sep = small]
			\mathsf{ish2adjl} \: f \ar[rr, leftrightarrow, "\simeq"] \ar[rd, leftrightarrow] 
				&& \mathsf{ish2adj} \: f \ar[ld, leftrightarrow] \\
			& \mathsf{adj} \: f \ar[ld, leftrightarrow] \ar[rd, leftrightarrow] & \\
			\mathsf{ishadjl} \: f \ar[rr, leftrightarrow, "\simeq"] \ar[rd, leftrightarrow] 
				&& \mathsf{ishadj} \: f \ar[ld, leftrightarrow] \\
			& \mathsf{qinv} f &
		\end{tikzcd}
	\]
	where rows 1 and 3 are propositions.
	
	As before, appending either one of these coherences yields a proposition, but appending both coherences
		yields a non-proposition once more. 
	\begin{defi}[\texttt{two\_adj/two\_adj}]
		Given $f : A \to B$, the structure of a \emph{2-adjoint equivalence} on $f$ is the type:
		\[
			\mathsf{2adj} \: f :\equiv \sum_{g : B \to A} \  \sum_{\eta : gf \sim \mathsf{id}_A} \ 
				\sum_{\varepsilon : fg \sim \mathsf{id}_B} \ 
				\sum_{\tau : f[\eta] \sim \varepsilon_f} \ 
				\sum_{\theta : \eta_g \sim g[\varepsilon]} \
				\mathsf{Coh} \: \eta \cdot g \llbracket \tau \rrbracket \sim \theta_f \times
				\tau_g \cdot \mathsf{Coh} \: \varepsilon \sim f \llbracket \theta \rrbracket.
		\]
	\end{defi}
	\begin{thm}[\texttt{two\_adj/two\_adj\_equiv\_pi\_refl\_eq}]
		Given $f : A \to B$ such that $\mathsf{ishadj}\: f$ is inhabited, we have
		\[
			\mathsf{2adj}\: f \simeq \prod_{x : A}
				 (\mathsf{refl}_{\mathsf{refl}_x} = \mathsf{refl}_{\mathsf{refl}_x}).
		\]
	\end{thm}
	\begin{proof}
		By Equivalence Induction \ref{T:equivInduction}, it suffices to show $\mathsf{2adj} \: \mathsf{id}_A 
			\simeq \prod\limits_{x : A} (\mathsf{refl}_{\mathsf{refl}_x} = \mathsf{refl}_{\mathsf{refl}_x})$.
		Observe that
		\begin{align}
			\mathsf{2adj}\: \mathsf{id}_A 
				&\equiv \sum_{g : A \to A} \  \sum_{\eta : g \sim \mathsf{id}_A} \ 
					\sum_{\varepsilon : g \sim \mathsf{id}_A} \ 
					\sum_{\tau : \mathsf{id}_A[\eta] \sim \varepsilon} \ 
					\sum_{\theta : \eta_{g} \sim g[\varepsilon]}
					\mathsf{Coh} \: \eta \cdot g \llbracket \tau \rrbracket \sim \theta 
					\times \tau_g \cdot \mathsf{Coh} \: \varepsilon \sim \mathsf{id}_A \llbracket \theta \rrbracket 
					\nonumber \\
			\label{eq:2adjprev} &\simeq \sum_{\theta : \mathsf{refl} \sim \mathsf{refl}}
				\mathsf{Coh} \: \mathsf{refl} \cdot \mathsf{id}_A \llbracket \mathsf{refl}_{\mathsf{refl}} \rrbracket
					\sim \theta
				\times \mathsf{refl}_{\mathsf{refl}} \cdot \mathsf{Coh} \: \mathsf{refl}
					\sim \mathsf{id}_A \llbracket \theta \rrbracket \\
			&\equiv \sum_{\theta : \mathsf{refl} \sim \mathsf{refl}}
				\mathsf{refl}_{\mathsf{refl}} \sim \theta
				\times \mathsf{refl}_{\mathsf{refl}} \sim \mathsf{id}_A \llbracket \theta \rrbracket \nonumber \\
			&\simeq \sum_{\theta : \mathsf{refl} \sim \mathsf{refl}}
				\mathsf{refl}_{\mathsf{refl}} \sim \theta
				\times \mathsf{refl}_{\mathsf{refl}} \sim \theta \nonumber \\
			&\simeq	\sum_{\theta : \mathsf{refl} \sim \mathsf{refl}} \ 
				\sum_{\mathcal{A} : \mathsf{refl}_{\mathsf{refl}} \sim \theta}
				\mathsf{refl}_{\mathsf{refl}} \sim \theta \nonumber \\
			&\simeq \sum_{u : \sum\limits_{\theta : \mathsf{refl} \sim \mathsf{refl}}
				\mathsf{refl}_{\mathsf{refl}} \sim \theta}
				\mathsf{refl}_{\mathsf{refl}} \sim \mathsf{pr}_1 \: u \nonumber \\
			\label{eq:2adjline} &\simeq \mathsf{refl}_{\mathsf{refl}} \sim \mathsf{refl}_{\mathsf{refl}} \\
			&\equiv \prod_{x : A} \mathsf{refl}_{\mathsf{refl}_x} = \mathsf{refl}_{\mathsf{refl}_x}. \nonumber
		\end{align}
		The equivalence (\ref*{eq:2adjprev}) is from Theorem \ref{T:adjequiv}; 
			we contract $(g, \eta, \varepsilon, \tau)$
			to $(\mathsf{id}_A, \mathsf{refl}, \mathsf{refl}, \mathsf{refl}_{\mathsf{refl}})$.
			The equivalence (\ref*{eq:2adjline}) is an application of Lemma \ref{T:isContrSigma}.
	\end{proof}
	Once again, this result implies any type with non-trivial $\pi_3$ may be used to construct non-trivial
		inhabitants of this type. We know $\pi_3(S^2) = \mathbb{Z}$, which proves:
	\begin{cor}
		The type $\mathsf{2adj} \: \mathsf{id}_{S^2}$ is not a proposition. \qed
	\end{cor}
	
	Proving this result using function extensionality directly and path induction requires
		an equivalence that writes homotopies as equalities.
		By modularizing the case of $\mathsf{qinv}$, similar to the analogous proof for $\mathsf{adj}$,
		this result may be proven by manipulating $\Sigma$-types and applying
		Lemma \ref{T:isContrSigma} three times, with the formal proof in the library being 44 lines of code.
		As with $\mathsf{adj}$, one would expect this approach to be 40 to 80 lines shorter than one which
		uses function extensionality directly.
		
	\section*{Acknowledgements.} This work was carried out when the first, second, and fourth authors were undergraduates at the University of Western Ontario and was supported through two Undergraduate Student Research Awards and a Discovery Grant, all funded by the Natural Sciences and Engineering Research Council (NSERC) of Canada. We thank NSERC for its generosity.
	
	\bibliographystyle{alpha}
	\bibliography{lmcs7007}

\end{document}